\documentclass{amsart}

\newtheorem{theorem}{Theorem}[section]
\newtheorem{lemma}[theorem]{Lemma}

\newtheorem{definition}[theorem]{Definition}

\newtheorem{remark}[theorem]{Remark}

\newtheorem{proposition}[theorem]{Proposition}
\newtheorem{corollary}[theorem]{Corollary} 
\newtheorem{notation}[theorem]{Notation}

\newcommand{\N}{\bf{N}}

\newcommand{\de}{{\delta}}
\renewcommand{\Pr}{{\rm{Prob}}}
\newcommand{\pf}{P_{\geq 4}}
\newcommand{\pe}{P_0^{(1)}}
\newcommand{\pee}{P_0^{(2)}}
\newcommand{\peee}{P_0^{(3)}}

\title{On Permutations of Order Dividing a Given Integer} 
\author{Alice C. Niemeyer \and Cheryl E. Praeger}
\address{School of Mathematics and Statistics, \\
University of Western  Australia,\\
Nedlands, WA 6907\\
Australia.}
\email{alice@maths.uwa.edu.au\\
praeger@maths.uwa.edu.au}

\date{31 March 2006.}


\begin{document}
\begin{abstract}
We give a detailed analysis of the proportion of elements in the
symmetric group on $n$ points whose order divides $m$, for $n$ sufficiently 
large and $m\geq n$ with  $m=O(n)$.
\end{abstract}

\maketitle

 \section{Introduction}
The study of orders of elements in finite symmetric groups goes back at 
least to the work of Landau~\cite[p. 222]{Landau09} who proved that the 
maximum order of an element of the symmetric group $S_n$ on $n$ points 
is $e^{(1+o(1))(n\log n)^{1/2}}$. 
Erd\H{o}s and Tur\'an took a probabilistic approach in their seminal work in the 
area, proving in  \cite{ErdosTuran65,ErdosTuran67} that, for a uniformly 
distributed random element $g\in S_n$, the random variable $\log|g|$ is 
normally distributed with mean $(1/2) \log^2n$ and standard deviation 
$\frac{1}{\sqrt{3}} \log^{3/2}(n)$. Thus most permutations in $S_n$ have order 
considerably larger than $O(n)$. Nevertheless, permutations of order
$O(n)$, that is, of order at most $cn$ for some constant $c$, have 
received some attention in the literature. Let $P(n,m)$ denote the 
proportion of permutations $g\in S_n$ which satisfy $g^m = 1$, that is to say, 
$|g|$ divides $m$. In 1952 Chowla, Herstein and Scott~\cite{Chowlaetal52} 
found a generating function and some recurrence relations for $P(n,m)$
for $m$ fixed, and
asked for its asymptotic behaviour for large $n$. Several years 
later, Moser and Wyman \cite{MoserWyman55,MoserWyman56} derived an 
asymptotic for $P(n,m)$, for a fixed prime number $m$,  expressing it as 
a contour integral. Then in 1986, Wilf~\cite{Wilf86} obtained explicitly the 
limiting value of $P(n,m)$ for an arbitrary fixed value of $m$ as 
$n\rightarrow\infty$, see also the paper \cite{Volynets} of Volynets.
Other authors have considered equations $g^m=h$, for a fixed integer
$m$ and $h\in S_n$, see
\cite{BouwerChernoff85,GaoZha,MineevPavlov76a,MineevPavlov76b}.  

However in many applications, for example in \cite{Bealsetal03}, 
the parameters $n$ and $m$ are linearly related, 
so that $m$ is unbounded as $n$ increases.
For the special case where $m=n$, Warlimont \cite{Warlimont78} showed in 1978 
that most elements $g\in S_n$ satisfying 
$g^n=1$ are $n$-cycles, namely he proved that $P(n,n)$, for $n$ 
sufficiently large, satisfies
\[
\frac{1}{n} + \frac{2c}{n^2} \le P(n,n) \le \frac{1}{n} + \frac{2c}{n^2} +
 O\left(\frac{1}{n^{3-o(1)}}\right)
\]
where $c =1$ if $n$ is even and
 $c=0$ if $n$ is odd. Note that the proportion of $n$-cycles in $S_n$ is
 $1/n$ and, if $n$ is even, the proportion of elements that are a product of 
two cycles of length $n/2$ is $2/n^2$. Warlimont's result proves in particular 
that most permutations 
satisfying $g^n=1$ are $n$-cycles. More precisely it implies that the 
conditional 
probability that a random element $g\in S_n$ is an $n$-cycle, given that $g^n
=1$, lies between $1-2c n^{-1} - O(n^{-2+o(1)})$ and $1-2c n^{-1} +
 O(n^{-2})$.  

The main results of this paper, Theorems~\ref{leadingterms} and \ref{bounds},
generalise Warlimont's result, giving a detailed analysis of $P(n,m)$ 
for large $n$, where $m=O(n)$ and $m\geq n$. For 
this range of values of $n$ and $m$, we have $rn\leq m<(r+1)n$ for some 
positive integer $r$, and we analyse $P(n,m)$ for $m$ in this range, for a 
fixed value of $r$ and $n\rightarrow\infty$. It turns out that the kinds of
elements that make the largest contribution to $P(n,m)$ depend heavily on the
arithmetic nature of $m$, for example, on whether $m$ is divisible by $n$
or by $r+1$.
We separate out several cases in the statement of our results.
Theorem~\ref{leadingterms} deals with two cases for which we give 
asymptotic expressions for $P(n,m)$. The first of these  reduces in the case 
$m=n$ to Warlimont's theorem~\cite{Warlimont78} (modulo a small discrepancy 
in the error term). For other values of $m$ lying strictly 
between $rn$ and $(r+1)n$ we 
obtain in Theorem~\ref{bounds} only an upper bound 
for $P(n,m)$, since the exact value depends on both the arithmetic nature 
and the size of $m$ (see also Remark~\ref{remark:leadinterms}).   
 
\begin{theorem}\label{leadingterms}
Let $n$ and $r$ be positive integers.
Then for a fixed value of $r$ and sufficiently large $n$, the following hold.
\begin{enumerate}
\item[(a)]  
$\displaystyle{
P(n,rn)=\frac{1}{n}+\frac{c(r)}{n^2}
+O\left(\frac{1}{n^{2.5-o(1)}}\right)
}$
where 
$c(r)=\sum
(1+\frac{i+j}{2r})$
and the sum is over all pairs $(i,j)$ such that $1\leq i,j\leq r^2, 
ij =r^2,$ and both $r+i, r+j$ divide $rn$.   
In particular $c(1)=0$ if $n$ is odd, and $2$ if $n$ is even.
\item[(b)] If $r=t!-1$ and $m=t!(n-t)=(r+1)n-t\cdot t!$, then 
\[
P(n,m)=\frac{1}{n}+\frac{t+c'(r)}{n^2}+O\left(\frac{1}{n^{2.5-o(1)}}
\right)
\]
where 
$c'(r)=\sum(1+\frac{i+j-2}{2(r+1)})$
and the sum is over all pairs $(i,j)$ such that $1< i,j\leq (r+1)^2, 
(i-1)(j-1) =(r+1)^2,$ and both $r+i, r+j$ divide $m$.   
\end{enumerate}
\end{theorem}

\begin{theorem}\label{bounds}
Let $n,m,r$ be positive integers such that $rn< m<(r+1)n$, and $\de$ 
a real number such that $0<\de\leq 1/4$.
Then for a fixed value of $r$ and sufficiently large $n$, 
\[
P(n,m)\leq \frac{\alpha.(r+1)}{m}+\frac{k(r)}
{n^2}+ O\left(\frac{1}{n^{2.5-2\de}}\right)
\]where $k(r) = \frac{4(r+3)^4}{r^2}$  and
\[
\alpha=\left\{\begin{array}{ll}
1&\mbox{if $r+1$ divides $m$ and $n-\frac{m}{r+1}
< \frac{m}{2(r+1)(r+2)-1}$}\\
0&\mbox{otherwise.}
	      \end{array}\right.
\]
\end{theorem}

\begin{remark}\label{remark:leadinterms}
{\rm
(a) In Theorem~\ref{leadingterms}(a), the 
leading term $1/n$ is the proportion of $n$-cycles, while the
proportion of permutations 
containing an $(n-t)$-cycle is $\frac{1}{n-t} = \frac{1}{n} +
\frac{t}{n^2} + O(\frac{1}{n^3})$, which contributes to the first
two terms in Theorem~\ref{leadingterms}(b).
 The terms
$\frac{c(r)}{n^2}$ and $\frac{c'(r)}{n^2}$ correspond to permutations 
in $S_n$ that have two long cycles, and these have lengths $\frac{m}
{r+i}$ and $\frac{m}{r+j}$, for some $(i,j)$ satisfying the conditions in 
Theorem~\ref{leadingterms} (a) 
or (b) respectively, (where $m=rn$ in part (a)).

(b) In Theorem~\ref{bounds}, if $r+1$ divides $m$ and $n-m/(r+1)<\frac{m}{2(r+1)(r+2)-1}$, 
then  the term $(r+1)/m$ comes from elements containing a cycle of length 
$m/(r+1)$. The term $\frac{k(r)}{n^2}$ corresponds to 
permutations with exactly two `large' cycles. More details are given in Remark~\ref{rem:general}.
}
\end{remark}

Our interest in $P(n,m)$ arose from algorithmic applications concerning finite 
symmetric groups.
 For example, $n$-cycles in $S_n$ satisfy the equation $g^n=1$, while elements 
 whose cycle structure consists of a 2-cycle and a single additional cycle 
of odd length $n-t$, where $t = 2$ or $3$,  satisfy  
 the equation $g^{2(n-t)} =1$. For an element $g$ of the latter type we can 
construct a transposition by forming the power $g^{n-t}$. In many cases the 
group $S_n$ is not given as a permutation group in its natural representation, 
and, while it is possible  to test whether an element $g$ satisfies one of 
these equations, it is often impossible to determine its cycle structure 
with certainty. It is therefore important to have lower bounds on the 
conditional probability that a random element $g$ has a desired cycle 
structure, given that it satisfies an appropriate equation. Using 
Theorem~\ref{leadingterms}, we obtained the following estimates of various 
conditional probabilities.

\begin{corollary}\label{cdnlprobs1}
Let $r, n$ be positive integers and let $g$ be a uniformly distributed random 
element of $S_n$. Then for a fixed value of $r$ and sufficiently large $n$, 
the following hold, where $c(r)$ and $c'(r)$ are as in 
Theorem~$\ref{leadingterms}$.
\begin{enumerate}
\item[(a)]  The conditional probability $P$ that $g$ is an $n$-cycle, given
that $|g|$ divides $rn$, satisfies
\begin{eqnarray*}
1-\frac{c(r)}{n}-O\left(\frac{1}
{n^{1.5-o(1)}}\right)&\leq& P
\leq 1-\frac{c(r)}{n}+O\left(\frac{1}
{n^{2}}\right).\\
\end{eqnarray*}
\item[(b)] If $r=t!-1$, then the conditional probability $P$ that $g$ 
contains an $(n-t)$-cycle, given that $|g|$ divides $t!(n-t)$, satisfies
\begin{eqnarray*}
1-\frac{c'(r)}{n}-O\left(\frac{1}
{n^{1.5-o(1)}}\right)&\leq& P
\leq 1-\frac{c'(r)}{n}+O\left(\frac{1}
{n^{2}}\right).\\
\end{eqnarray*}
\end{enumerate}
\end{corollary}

We note that Theorem~\ref{leadingterms} improves the upper bound of
$(1+o(1))/n$ obtained in \cite[Theorem 3.7]{Bealsetal03}, while 
Corollary~\ref{cdnlprobs1} improves the corresponding lower bound of 
$1-o(1)$ of  \cite[Theorem 1.3(a)]{Bealsetal03}. These results have been 
developed and refined further in \cite{NiemeyerPraeger05b} to derive explicit 
`non-asymptotic' bounds that hold for all $n$ and can be applied directly 
to improve the recognition algorithms for $S_n$ and 
$A_n$ in \cite{Bealsetal03}.

\bigskip\noindent{\bf Commentary on our approach}

Warlimont's proof in \cite{Warlimont78} of an upper bound for $P(n,n)$
and the  proof of \cite[Theorem 3.7]{Bealsetal03} by  Beals and Seress
of an  upper bound  for $P(n,m)$  for certain values  of $m$,  rely on
dividing the elements  of $S_n$ into disjoint unions  of smaller sets.
Warlimont divides the elements according  to how many `large' cycles a
permutation  contains. Fix  a real  number $s$  such that  $1/2 <  s <
1$. We say that a cycle  of a permutation in $S_n$ is \emph{$s$-small}
if its  length is  strictly less than  $n^s$, and  is \emph{$s$-large}
otherwise.  Beals  and Seress divide  the elements according  to the 
number of cycles in which three specified  points lie. Both
strategies are sufficient to  prove Warlimont's result or the slightly
more  general results  of  \cite[Theorem 3.7]{Bealsetal03}.   However,
neither is sufficient  to prove the general results  in this paper. In
particular, Warlimont's  approach breaks down when  trying to estimate
the proportion of  elements with no or only one  large cycle, which is
perhaps   why   no   progress   has   been  made   since   his   paper
\cite{Warlimont78}  towards  answering  Chowla, Herstein  and  Scott's
original question about the asymptotic behaviour of $P(n,m)$ for large
$n$.  One of  the key ideas that allowed  us to generalise Warlimont's
work is the  insight that the number of  permutations which contain no
$s$-large cycles  can be estimated  by considering their  behaviour on
three  specified points.   Another important  strategy is  our careful
analysis of elements containing only one large cycle by separating out
divisors of $m$ which are very close to $n$.

We regard Theorem~\ref{lem:props} below as the main outcome of the first stage 
of our analysis. It is used in the proof of Theorem~\ref{leadingterms}.
The statement of Theorem~\ref{lem:props} involves the number $d(m)$ of 
positive divisors of $m$, and the fact that $d(m)=m^{o(1)}$, see 
Notation~\ref{notation}~(c). It estimates the 
proportion $P_0(n,m)$ of elements of $S_n$ of order dividing $m$ and having no 
$s$-large cycles.

\begin{theorem}\label{lem:props}
Let $n,m$ be positive integers such that $m\geq n$, and let $s$ be a positive 
real number such that $1/2<s<1$. Then, with $P_0(n,m)$ 
as defined above, there is a constant $c$ such that
\[
P_0(n,m)<\frac{c d(m)m^{2s}}{n^3}=O\left(\frac{m^{2s+o(1)}}{n^3}\right).
\]
\end{theorem}

Theorem~\ref{lem:props} is proved in Section~\ref{sec:proportions} and
the other results are proved in Section~\ref{sec:stheo}. 

\section{Proof of Theorem~\ref{lem:props}}\label{sec:proportions}

In this section we introduce some notation that will be used throughout 
the paper, and we prove Theorem~\ref{lem:props}.  
Note that the order $|g|$ of a permutation $g \in S_n$
divides $m$ if and only if the length of each cycle of $g$
divides $m$.  
 Thus $P(n,m)$ is the proportion of elements in $S_n$ all of
whose cycle lengths divide $m$. 
As indicated in the introduction, we estimate $P(n,m)$ by partitioning this
proportion in various ways. Sometimes the partition is according to the number 
of large cycle 
lengths, and at other times it is defined in terms of the cycles containing 
certain points. We specify these partitions, and give some other 
notation, below.

\begin{notation}\label{notation}{\rm
 The numbers $n,m$ are positive integers, and the symmetric group $S_n$ acts
naturally on the set $\Omega=\{1,2,\dots,n\}$. 
\begin{enumerate}
\item[(a)]  $s$ is a real number such that $1/2 < s < 1$.  A divisor $d$ of 
$m$ is said to be  $s$-\emph{large} or $s$-\emph{small} if  $d \geq m^{s}$ 
or $d < m^s$,  respectively; $D_\ell$ and $D_s$ denote the sets of all 
$s$-large and $s$-small divisors $d$ of $m$, respectively,  such that 
$d \le n$.
\item[(b)] For $g\in S_n$ with order dividing $m$, a $g$-cycle of length $d$ 
is called  $s$-\emph{large} or $s$-\emph{small} according as $d$ is an 
$s$-large or $s$-small divisor of $m$. 
\item[(c)] $d(m)$ denotes the number of positive divisors of
$m$ and $\delta$ and $c_\delta$ are positive real numbers such that
  $\delta < s$ and  $d(m) \le c_\delta m^{\delta}$ 
for all $m \in \N$.
\item[(d)] The following functions of $n$ and $m$ denote the proportions of
elements $g\in S_n$ of order dividing $m$ and 
satisfying the additional properties given in the last column of 
the table below. 

\noindent
\begin{center}
\begin{tabular}{|ll|}
\hline
$P_0(n,m)$ &\quad all  $g$-cycles are $s$-small\\
$\pe(n,m)$ &\quad all  $g$-cycles are $s$-small and\\
             &\quad $1,2,3$ lie in the same $g$-cycle, \\
$\pee(n,m)$ &\quad  all  $g$-cycles are $s$-small and\\ 
             &\quad $1,2,3$ lie in exactly two $g$-cycles\\
$\peee(n,m)$ &\quad   all  $g$-cycles are $s$-small and\\
             &\quad $1,2,3$ lie in three different $g$-cycles\\
$P_1(n,m)$ &\quad  $g$ contains exactly one $s$-large cycle\\
$P_2(n,m)$ &\quad  $g$ contains exactly two $s$-large cycles\\
$P_3(n,m)$ &\quad  $g$ contains exactly three $s$-large cycles \\
$\pf(n,m)$ &\quad  $g$ contains at least four $s$-large cycles \\
\hline
\end{tabular}
\end{center}
\end{enumerate}
}
\end{notation}
\bigskip

With respect to part (c) we note, 
see \cite[pp. 395-396]{NivenZuckermanetal91}, that for each $\delta >
0$ there exists a constant $c_\delta > 0$ such that $d(m) \le c_\delta
m^\delta$ for all $m \in \N.$ This means that the parameter $\delta$
can be any positive real number and in particular that $d(m) = m^{o(1)}.$

\bigskip
\noindent
Note that 
\begin{equation}\label{eq-pi}
P_0(n,m) = \pe(n,m) + \pee(n,m) + \peee(n,m)
\end{equation}
and 
\begin{equation}\label{eq-qi}
P(n,m) = P_0(n,m) + P_1(n,m) + P_2(n,m) + P_3(n,m)+\pf(n,m).
\end{equation}
We begin by deriving recursive expressions for the $P_0^{(i)}(n,m)$.

\begin{lemma}\label{lem:theps}
Using Notation~$\ref{notation}$, the following hold, where we take 
$P_0(0,m) = 1.$
\begin{enumerate}
\item[(a)] $\displaystyle{\pe(n,m) = \frac{(n-3)!}{n!} 
\sum_{d \in D_s,\ d\ge 3}{(d-1)(d-2)}P_0(n-d,m),}$
\item[(b)]  $\displaystyle{
\pee(n,m) = \frac{3(n-3)!}{n!}\sum_{\stackrel{d_1, d_2 \in D_s }{2\le
d_2,\ d_1+d_2\le n}}  (d_2-1)P_0(n-d_1-d_2,m)}$, 
\item[(c)] $\displaystyle{
\peee(n,m) = \frac{(n-3)!}{n!} \sum_{\stackrel{d_1,d_2,d_3\in D_s
    }{d_1+d_2+d_3  \le n}}  
P_0(n-d_1-d_2 -d_3,m)}$. 
\end{enumerate}
\end{lemma}

\begin{proof}
We first compute $\pe(n,m)$, the proportion of those permutations $g\in S_n$
of order dividing $m$ with all cycles $s$-small, for which the points
 $1, 2, 3$ are contained in one $g$-cycle, $C$ say, of 
length $d$ with $d \in D_s$ and $d\geq 3.$
We can choose the remainder of the  support set of $C$ in 
$\binom{n-3}{d-3}$ ways and then 
the cycle $C$ in $(d-1)!$ ways. 
The rest of the
permutation $g$ can be chosen in $P_0(n-d,m)(n-d)!$ ways. Thus,
for a given 
$d$, the number of such elements is
$(n-3)!(d-1)(d-2)P_0(n-d,m)$. We obtain
the proportion $\pe(n,m)$ by summing over all  
$d\in D_s$ with $d\geq3$, and then dividing by $n!$, 
so part (a) is proved. 

Next we determine the proportion $\pee(n,m)$  of those permutations $g\in S_n$
of order dividing $m$  with all cycles $s$-small, 
for which one of the points  $1, 2, 3$  is contained 
in a $g$-cycle $C_1$, and the other two of these points are contained in a 
different $g$-cycle $C_2$. Let $d_1$ and $d_2$ denote the lengths of the 
cycles $C_1$ and $C_2$, respectively, so $d_1, d_2\in D_s$ and $d_2 \ge 2.$
Firstly we choose the support set of $C_1$ in $\binom{n-3}{d_1-1}$ ways and 
the cycle $C_1$ in $(d_1-1)!$ ways. 
Secondly we choose the support set of $C_2$ in $\binom{n-d_1 -2}{d_2-2}$ ways and 
the cycle $C_2$ in $(d_2-1)!$ ways. 
Finally, the rest of the
permutation $g$ is chosen in $P_0(n-d_1
-d_2,m)(n-d_1-d_2)!$ ways. 
Thus, for a given pair $d_1, d_2$, the number of these elements is 
$(n-3)!(d_2-1)P_0(n-d_1-d_2,m)$. 
Since there are three choices for $C_1\cap\{ 1, 2, 3\}$, we have 
\begin{eqnarray*}
\pee(n,m) & = & \frac{3(n-3)!}{n!}\sum_{\stackrel{d_1, d_2 \in D_s}{2\le
    d_2,\ d_1+d_2 \le n}} (d_2-1) P_0(n-d_1-d_2,m). \\ 
\end{eqnarray*}
Finally we consider the proportion  $\peee(n,m)$ 
 of those permutations $g\in S_n$ of order dividing $m$ with all cycles 
$s$-small, for which each one of the
points $1,  2,  3$ 
is  contained in  a separate $g$-cycle, say $C_i$ contains $i$ and
$C_i$ has length $d_i \in D_s$.  We can
choose, in order, the support set 
of  $C_1$  in  $\binom{n-3}{d_1-1}$  ways  and  the  cycle  $C_1$  in
$(d_1-1)!$ ways, the support set of $C_2$ in $\binom{n-d_1 -2}{d_2-1}$
ways and  the cycle $C_2$  in $(d_2-1)!$ ways, the support  set of
$C_3$ in  $\binom{n-d_1 -d_2 -1}{d_3-1}$  ways and the  cycle $C_3$ in
$(d_3-1)!$ ways, and the rest of the permutation in
$P_0(n-d_1-d_2-d_3,m)(n-d_1-d_2-d_3)!$ ways. The expression for $\peee(n,m)$ 
in part (c) now follows.
\end{proof}

Next we derive 
expressions for the $P_i(n,m)$ and $\pf(n,m)$.

\begin{lemma}\label{lem:qi}
Using Notation~$\ref{notation}$, and writing $P_0(0,m)=1$,
\begin{enumerate}
\item[(a)] ${\displaystyle P_0(n,m) = \frac{1}{n}\sum_{d\in D_s}
P_0(n-d, m),}$
\item[(b)] ${\displaystyle P_1(n,m) = \sum_{d\in D_\ell }
\frac{1}{d} P_0(n-d, m)},$  
\item[(c)] ${\displaystyle P_{2}(n,m) = \frac{1}{2} \sum_{d_1, d_2\in D_\ell }
\frac{1}{d_1d_2} P_0(n-d_1-d_2, m)},$
where the sum is over all ordered pairs $(d_1, d_2)$ with $d_1 + d_2
\le n$.
\item[(d)] ${\displaystyle P_3(n,m) = \frac{1}{6}\sum_{d_1, d_2, d_3
    \in D_\ell} 
\frac{1}{d_1d_2d_3} P_0(n-d_1-d_2 - d_3, m)}$, where the sum  is over all 
ordered triples $(d_1,d_2,d_3)$ with $d_1 + d_2 + d_3 \le n$.
\item[(e)] ${\displaystyle \pf(n,m) \leq 
\frac{1}{24}\sum_{d_1, d_2, d_3,d_4 \in D_\ell}
\frac{1}{d_1d_2d_3d_4} P(n-d_1-d_2 - d_3-d_4, m)}$,
where the sum is over all ordered $4$-tuples
$(d_1,d_2,d_3,d_4)$
with $d_1 + d_2 + d_3+d_4 \le n$.
\end{enumerate}
\end{lemma}

\begin{proof} For each permutation in $S_n$ of order dividing $m$ 
and all cycles $s$-small, the point 1 lies in a cycle of length $d$,
for some $d\in D_s$. For this value of $d$ there are $\binom{n-1}
{d-1}(d-1)!$ choices of $d$-cycles containing 1, and $P_0(n-d,m)(n-d)!$
choices for the rest of the permutation. Summing over all $d\in D_s$ 
yields part (a).

The proportion of permutations in $S_n$ of order dividing $m$ and having
exactly one $s$-large cycle of length $d$ is $\binom{n}{d}(d-1)! P_0(n-d,m) 
(n-d)!/n!$. Summing over all $d\in D_\ell$ yields part (b).

In order  to find the  proportion of elements in $S_n$ of order dividing 
$m$ and having exactly two $s$-large  cycles we
count triples $(C_1, C_2, g)$,  where $C_1$ and $C_2$ are cycles of
lengths $d_1$ and $d_2$ respectively, $d_1, d_2\in D_\ell$, 
$g\in S_n$ has order dividing $m$, $g$ contains $C_1$ and $C_2$ in its  
disjoint cycle representation, and all other $g$-cycles are $s$-small.
For a given $d_1, d_2$, we have $\binom{n}{d_1}(d_1-1)!$
choices for $C_1$, then $\binom{n-d_1}{d_2}(d_2-1)!$ choices for
$C_2$, and then the rest of the element $g$ containing $C_1$ and $C_2$ can
be chosen in $P_0(n-d_1-d_2,m)(n-d_1-d_2)!$ ways. Thus the ordered pair 
$(d_1,d_2)$ contributes $\frac{n!}{d_1d_2}P_0(n-d_1-d_2,m)(n-d_1-d_2)!$ 
triples, and each element $g$ with the 
properties required for part (c) contributes exactly two of these triples.
Hence, summing over ordered pairs $d_1, d_2\in D_\ell$  yields (c).

Similar counts are used for parts (d) and (e). For $P_3(n,m), \pf(n,m)$ 
we count  4-tuples $(C_1, C_2,C_3, g)$ and $5$-tuples $(C_1,C_2,C_3,C_4,g)$ 
respectively,
such that, for each $i$, $C_i$ is a cycle of length $d_i$ for some 
$d_i\in D_\ell$, $g\in S_n$ has order dividing $m$, and $g$ contains 
all the cycles $C_i$ in its disjoint cycle representation. 
The reason we have an 
inequality for $\pf(n,m)$ is that in this case each $g$ occurring has at 
least four $s$-large cycles 
 and hence occurs in at least 24 of the 5-tuples, but possibly more.
\end{proof}

We complete this section by giving a proof of Theorem~\ref{lem:props}. The 
ideas for its proof were developed from arguments in Warlimont's paper 
\cite{Warlimont78}.

\begin{lemma}\label{newPs} Let $m\geq n\geq3$, and let $s, \de$ be as in 
Notation~{\rm\ref{notation}}. Then 
\[
P_0(n,m) < \frac{(1 + 3c_\delta + c_\delta^2)d(m)m^{2s}}{n(n-1)(n-2)}<
\frac{c'd(m)m^{2s}}{n^3}= O\left(\frac{m^{2s+\delta}}{n^3}\right)
\]
where, if $n\geq6$, we may take
\[
c'=\left\{\begin{array}{ll}
2(1 + 3c_\delta + c_\delta^2)&\mbox{for any $m\geq n$}\\
10&\mbox{if $m\geq c_\delta^{1/(s-\delta)}$.}
          \end{array}\right.
\]
In particular Theorem~{\rm\ref{lem:props}} is true. Moreover, if  
in addition $n\geq m^s+cn^a$ for some positive constants $a,c$ with $a\leq 1$, 
then $P_0(n,m)=O\left(\frac{m^{2s+2\de}}{n^{1+3a}}\right)$.
\end{lemma}

\begin{proof}
First assume only that $m\geq n\geq3$.  Let $D_s$, and  
$P_0^{(i)}(n,m)$, for $i = 1, 2, 3$, be as in Notation~\ref{notation}. By 
(\ref{eq-pi}), $P_0(n,m)$ is the sum of the  $P_0^{(i)}(n,m)$. 
We first estimate $\pe(n,m).$
By Lemma~\ref{lem:theps}~(a), and using the fact that $d<m^s$ for 
all $d\in D_s$, 
\[\pe(n,m) \le\frac{(n-3)!}{n!} 
\sum_{\stackrel{d \in D_s}{d\ge 3}}{(d-1)(d-2)}< 
\frac{d(m) m^{2s}}{n(n-1)(n-2)}.
\]
Similarly, by Lemma~\ref{lem:theps}~(b), 
\begin{eqnarray*}
\pee(n,m) & <  & \frac{3(n-3)!}{n!}\sum_{d_1, d_2 \in D_s} (d_2-1) 
\le   \frac{3d(m)^2m^{s}}{n(n-1)(n-2)}
\end{eqnarray*}
and by Lemma~\ref{lem:theps}~(c), 
\begin{eqnarray*}
\peee(n,m) &<& \frac{(n-3)!}{n!} \sum_{d_1,d_2,d_3\in D_s} 1
\le  \frac{d(m)^3}{n(n-1)(n-2)}.\\
\end{eqnarray*}

Thus, using the fact noted in 
Notation~\ref{notation} that $d(m) \le  c_\delta m^\delta$,
\begin{eqnarray*}
P_0(n,m) & \le  & 
\frac{d(m) \left( m^{2s} +3d(m)m^{s} + d(m)^2\right)
}{n(n-1)(n-2)} \\
&\le&\frac{d(m)m^{2s}\left( 1 +3c_\delta m^{\delta-s} + (c_\delta m^{\delta-s})^2\right)}{
n(n-1)(n-2)}< \frac{c'd(m) m^{2s}}{n^3}.
\end{eqnarray*}
To estimate $c'$ note first that, for $n\geq6$, $n(n-1)(n-2)> n^3/2$. Thus
if $n\geq6$ then, for any $m\geq n$ we may take $c'= 2(1 + 3c_\delta + c_\delta^2).$
If  $m\geq c_\delta^{1/(s-\delta)}$, then $c_\delta m^{\delta-s}\leq 1$ and so we may take $c'=10$.
Theorem~\ref{lem:props} now follows since $d(m)=m^{o(1)}$. 
Now assume that $n\geq m^s+cn^a$ for some positive constants $c$ and $a$.  
By Lemma~\ref{lem:qi},
\[
P_0(n,m)= \frac{1}{n}\sum_{d\in D_s}P_0(n-d, m).
\]
For each $d\in D_s$ we have $m>n-d\geq n-m^s\geq cn^a$, and hence applying 
Theorem~\ref{lem:props} (which we have just proved),  
\[P_0(n-d,m) < \frac{c'd(m)m^{2s}}{(n-d)^3}
\leq \frac{c'd(m) m^{2s}}{c^3 n^{3a}}.
\]
Thus, $P_0(n,m) \leq  \frac{d(m)}{n} \left(\frac{c'd(m)m^{2s}}{c^3n^{3a}}
\right)\le \frac{c'c_\delta^2m^{2s + 2\delta}}{c^3n^{1+3a}}$.
\end{proof}

\section{Proof of Theorem~\ref{leadingterms}}\label{sec:stheo}

First we determine the `very large' divisors of $m$ that are at most $n$.

\begin{lemma}\label{lem:divat}
Let $r, m$ and $n$  be positive integers such that
$rn\le m < (r+1)n$. 
\begin{enumerate}
\item[(a)] If  $d$ is a divisor of $m$ such that
$d \le n$, then one of the following holds:
\begin{enumerate}
\item[(i)] $d=n = \frac{m}{r}$,
\item[(ii)] $d = \frac{m}{r+1}$ so that $\frac{r}{r+1}n \le d < n$,
\item[(iii)] $d \le \frac{m}{r+2}<\frac{r+1}{r+2}n$.
\end{enumerate}
\item[(b)]
Moreover, if $d_1, d_2$ are divisors of
$m$ for which 
\[
d_1\le d_2 \le   \frac{m}{r+1}\quad  \mbox{and}\quad 
n \ge d_1 + d_2 > \frac{m(2r+3)}{2(r+1)(r+2)},
\] 
then $d_1=\frac{m}{c_1}, d_2=
\frac{m}{c_2}$, where $c_1, c_2$ divide $m$, and satisfy
$c_2 \le 2r+3$, and either $r+2\leq c_2 \le c_1 < 2(r+1)(r+2)$, 
or $c_2=r+1$, $c_1\geq r(r+1)$.
\end{enumerate}

\end{lemma} 

\begin{proof}  
As $d$ is a divisor of $m$ there is a positive integer $t$ such 
that $d = \frac{m}{t}$. Now
$\frac{m}{t} \le n \le \frac{m}{r}$ and therefore $r \le t.$
If $r = t$ then $r$ divides $m$ and
$d = \frac{m}{r} \le n$,  and since also $rn \le m$ it follows that
$d = \frac{m}{r}=n$ and 
(i) holds. If $t \ge r+2$ then (iii) holds. Finally, if $t=r+1$, then
$d = \frac{m}{r+1}$ and $\frac{r}{r+1}n \le \frac{m}{r+1} < n$ and
hence (ii) holds.

Now we prove the last assertion. Suppose that $d_1, d_2$ are divisors of $m$ 
which are at most $ \frac{m}{r+1}$, and such that 
$d_1\leq d_2$ and $n\geq d_1 + d_2 > \frac{m(2r+3)}{2(r+1)(r+2)}$. Then, as $d_1,
d_2$ divide $m$, there are integers $c_1, c_2$
such that 
$d_1 = m/c_1$ and $d_2 = m/c_2.$
Since $d_i \le  m/(r+1)$ we have
$c_i \ge r+1$ 
for $i = 1,2$, and since $d_1\le d_2$ we have $c_1\ge c_2$.
Now $m/r \ge n \ge d_1 + d_2 >  \frac{m(2r+3)}{2(r+1)(r+2)}$, and hence
$1/r \ge 1/c_1 + 1/c_2 >  \frac{2r+3}{2(r+1)(r+2)}$. 
If  $c_2 \ge 2(r+2)$ then, as $c_1\ge c_2$, we would have 
$1/c_1 + 1/c_2 \le 1/(r+2)$, 
which is not the case.  Thus  $r+1 \le c_2 \le 2r+3.$ 
If $c_2\geq r+2$, then
\[
\frac{1}{c_1}>  \frac{2r+3}{2(r+1)(r+2)} - \frac{1}{c_2} \ge 
 \frac{2r+3}{2(r+1)(r+2)} - \frac{1}{r+2} = 
 \frac{1}{2(r+1)(r+2)}\]  and hence
$c_1 < 2(r+1)(r+2)$ as in the statement. 
On the other hand, if $c_2=r+1$, then 
\[
\frac{1}{c_1}\leq \frac{n}{m}-\frac{1}{c_2}\leq \frac{1}{r}-\frac{1}{r+1}=\frac{1}{r(r+1)}
\]
so $c_1\geq r(r+1)$.
\end{proof}

The next result gives our first estimate of an 
upper bound for the proportion $P(n,m)$ of elements in $S_n$ of order
dividing $m$.  Recall our observation that the parameter $\delta$ in
Notation~\ref{notation}(c) can be any positive real number; in
Proposition~\ref{prop:general} we will restrict to $\delta \le s-\frac{1}{2}.$
Note that the requirement $rn\leq m<(r+1)n$ implies that 
$\frac{n}{r+1}\leq n-\frac{m}{r+1}\leq \frac{m}{r(r+1)}$; the first case of 
Definition~\ref{def:kr}~(b) below requires an upper bound of approximately 
half this quantity.

\begin{definition}\label{def:kr}
Let  $r,\, m,\, n$ be positive integers such that
$rn\le m <  (r+1)n$. Let $1/2<s\leq 3/4$ and $0<\de\leq s-\frac{1}{2}$.
\begin{itemize}
\item[(a)] Let $\alpha = \begin{cases}  1 & \mbox{if\ } m=rn,\\
                                        0 & \mbox{otherwise.}
			 \end{cases}$

\item[(b)] Let $\alpha' = \begin{cases}  1 & \mbox{if\ } (r+1) \mbox{\
 divides\ } m \ 
 \mbox{and\ }n-\frac{m}{r+1}<\frac{m}{2(r+1)(r+2)-1}, \\
                                        0 & \mbox{otherwise.}
			 \end{cases}$
\item[(c)] Let $t(r,m,n)$ denote  the number of divisors $d$ of $m$
  with $\frac{m}{2r+3} \leq d\leq\frac{m}{r+1}$ such that there exists
  a divisor $d_0$ of $m$ satisfying
\begin{itemize}
\item[(i)] $d+d_0\leq n$ and
\item[(ii)] $\frac{m}{2(r+1)(r+2)}< d_0\leq d$.
\end{itemize}
\item[(d)] Let $k(r,m,n)=t(r,m,n)\frac{2(r+1)(r+2)(2r+3)}{r^2}.$ 
\end{itemize}
\end{definition}

\begin{proposition}\label{prop:general}
Let  $r,\, m,\, n, s$ and $\delta$ be as in Definition~{\rm\ref{def:kr}}.
Then, for a fixed value of $r$ and sufficiently large $n$, 
\[ 
P(n,m) \le \frac{\alpha}{n}+\frac{\alpha'.(r+1)}{m}+\frac{k(r,m,n)}{n^2}+
O\left(\frac{1}{n^{1+2s-2\de}}
\right),
\]
where $\alpha, \alpha', t(r, m, n)$ and $k(r, m, n)$ are as in
Definition~$\ref{def:kr}.$  
Moreover,   $t(r,m,n) \le r+3$ and $k(r,m,n) \le 
  \frac{4(r+3)^4}{r^2} $.
\end{proposition}

\begin{remark}\label{rem:general}{\rm 
(a) The term $\frac{1}{n}$, which occurs if and only if $m=rn$, corresponds to 
the $n$-cycles in $S_n$, and is the exact proportion of these elements.
We  refine the estimate for $P(n,rn)$ in Theorem~\ref{rn} below.

(b) The term $\frac{r+1}{m}$, which occurs only if $r+1$ divides $m$ and 
$n-\frac{m}{r+1}<\frac{m}{2(r+1)(r+2)}$, corresponds to permutations with order 
dividing $m$ and having either one or two 
$s$-large cycles, with one (the larger in the case of two cycles) of length $\frac{m}{r+1}$. 
The proportion 
of elements of $S_n$ containing a cycle of length $\frac{m}{r+1}$ 
is $\frac{r+1}{m}$, and if there exists a positive integer 
$d\leq n-\frac{m}{r+1}$ such that $d$ does not divide $m$, 
then some of these elements have a $d$-cycle and hence
do not have order dividing $m$. Thus $\frac{r+1}{m}$ may be 
an over-estimate for the proportion of elements in $S_n$ (where 
$n-\frac{m}{r+1}<\frac{m}{2(r+1)(r+2)}$) having order dividing $m$, having 
exactly one $s$-large cycle of length $\frac{m}{r+1}$, and possibly one additional 
$s$-large cycle of length dividing $m$.
However it is difficult to make a more precise estimate 
for this term that holds for all sufficiently large $m,n$. In  
Theorem~\ref{rn} we treat some special cases
where this term either does not arise, or can be determined precisely.

(c)   The   term  $\frac{k(r,m,n)}{n^2}$   arises   as  follows   from
permutations  that  have  exactly  two  $s$-large  cycles  of  lengths
dividing $m$.  For each of  the $t(r,m,n)$ divisors  $d$ of $m$  as in
Definition~\ref{def:kr}(c), let  $d_0(d)$ be  the largest of  the divisors
$d_0$   satisfying   Definition~\ref{def:kr}(c)(i),(ii).   Note that  $d_0(d)$
depends  on  $d$.  Then  $k(r,m,n)/n^2$  is an  upper  bound  for  the
proportion  of  permutations of  order  dividing  $m$  and having  two
$s$-large cycles of lengths $d$ and  $d_0(d)$, for some $d$ satisfying 
 $\frac{m}{2r+3} \leq d\leq\frac{m}{r+1}$.  As  in (b) this term may  be 
an over-estimate, not  only for the
reason  given there,  but  also  because lower  bounds  for the  cycle
lengths  $d,  d_0(d)$ were  used  to  define  $k(r,m,n)$. Indeed in the case
$m=rn$ we are able to obtain the exact value of the coefficient of the
$\frac{1}{n^2}$ summand.  }
\end{remark}

\begin{proof} 
We  divide  the   estimation  of  $P(n,m)$  into  five  subcases.  Recall that,
 by (\ref{eq-qi}), $P(n,m)$ is the sum of  $\pf(n,m)$ and the $P_i(n,m)$, for 
$i=0,1,2,3$, where these are as defined in Notation~\ref{notation}.
We will use the recursive
formulae for $\pf(n,m)$ and the $P_i(n,m)$ in Lemma~\ref{lem:qi}, together 
with the expressions for $P_0(n,m)$ in  Theorem~\ref{lem:props} and 
Lemma~\ref{newPs}, to estimate these five quantities. Summing these estimates 
will give, by (\ref{eq-qi}), our estimate for $P(n,m)$.
We also use the information about divisors of $m$ in
Lemma~\ref{lem:divat}.

First we deal with $P_0(n,m)$. Since $r$ is fixed, it follows that, for 
sufficiently large $n$ (and hence sufficiently large $m$), we  have $m^s
\leq \frac{m}{r+2}$, which is less than $\frac{(r+1)n}{r+2}=n-\frac{n}{r+2}$. 
Thus $n>m^s+\frac{n}{r+2}$,  and applying Lemma~\ref{newPs} with $a=1, c=\frac{1}{r+2}$, it follows that
\[
P_0(n,m)=O\left(\frac{m^{2s+2\de}}{n^4}\right)=O\left(\frac{1}{n^{4-2s-
2\de}}\right)\leq O\left(\frac{1}{n^{1+2s-2\de}}\right)
\]
since $4-2s-2\de\geq 1+2s-2\de$ when $s\leq 3/4$.

Next we estimate $P_3(n,m)$ and $\pf(n,m)$. 
By  Lemma~\ref{lem:qi}, the latter satisfies
$\pf(n,m)\leq \frac{1}{24}\sum\frac{1}{d_1d_2d_3d_4}$, where 
the summation is over all ordered 4-tuples of $s$-large divisors of $m$ 
whose sum 
is at most $n$. Thus  $\pf(n,m)\leq \frac{1}{24}\,\frac{d(m)^4}{m^{4s}}=
O\left(\frac{1}{n^{4s-4\de}}\right)$. Also 
\[
P_3(n,m)= \frac{1}{6}\sum
\frac{1}{d_1d_2d_3}P_0(n-d_1-d_2-d_3,m),
\] 
where 
the summation is over all ordered triples of $s$-large divisors of $m$ whose 
sum 
is at most $n$. For such a triple $(d_1,d_2,d_3)$, if each $d_i\leq\frac{m}
{4(r+1)}$, then $n-\sum d_i\geq n-\frac{3m}{4(r+1)}>\frac{n}{4}$, and so by 
Lemma~\ref{newPs}, $P_0(n-\sum d_i,m)=O\left(\frac{m^{2s+\de}}{n^{3}}
\right)$. Thus the contribution of triples of this type to $P_3(n,m)$ is 
at most $O\left(\frac{d(m)^3m^{2s+\de}}{m^{3s}n^3}
\right)=O\left(\frac{1}{n^{3+s-4\de}}\right)$. For each of the
remaining triples, the   
maximum $d_i$ is greater than $\frac{m}{4(r+1)}$ and in particular there is a 
bounded number of choices for the maximum $d_i$. Thus the contribution of the 
remaining triples  to $P_3(n,m)$ is at most $O\left(\frac{d(m)^2}{m^{1+2s}}
\right)=O\left(\frac{1}{n^{1+2s-2\de}}\right)$. It follows that 
\[
P_3(n,m)+\pf(n,m)=O\left(\frac{1}{n^{x_3}}\right),
\] 
where $x_3=\min\{4s-4\de,3+s-4\de,1+2s-2\de
\}=1+2s-2\de$ (using the fact that $\de\leq s-\frac{1}{2}\leq \frac{1}{4}$). 

Now we estimate $P_2(n,m)$. By Lemma~\ref{lem:qi}, 
\[
P_{2}(n,m)= \frac{1}{2}\sum
\frac{1}{d_1d_2}P_0(n-d_1-d_2,m),
\] 
where 
the summation is over all ordered pairs of $s$-large divisors of $m$ whose sum 
is at most $n$. We divide these pairs $(d_1,d_2)$ into two subsets.  The first 
subset consists of those for which $n- d_1-d_2\geq n^\nu$, where
$\nu=(1+2s+\de)/3$. Note that $\nu<1$ since $\nu\leq s -\frac{1}{6}<1$ (because 
$\de\leq s-\frac{1}{2}$ and $s\leq \frac{3}{4}$).
For a pair  $(d_1,d_2)$ such that $n- d_1-d_2\geq n^\nu$, by 
Lemma~\ref{newPs}, $P_0(n-d_1-d_2,m)=O\left(\frac{m^{2s+\de}}{n^{3\nu}}
\right)$. 
Thus the total contribution  to $P_{2}(n,m)$ from pairs of this type is 
at most $O\left(\frac{d(m)^2m^{2s+\de}}{m^{2s}n^{3\nu}}
\right)=O\left(\frac{1}{n^{3\nu-3\de}}\right)=O\left(\frac{1}{n^{1+2s-2\de}}
\right)$. 

Now consider pairs  $(d_1,d_2)$ such that $n- d_1-d_2< n^\nu$. Since each
$d_i<n\leq m/r$, it follows that each $d_i\leq m/(r+1)$.
Since $\nu<1$, for sufficiently 
large $n$ (and hence sufficiently large $m$) we have $n^\nu\leq \left(\frac{m}{r}
\right)^\nu<\frac{m}{2(r+1)(r+2)}$. Thus, for each of the pairs $(d_1,d_2)$ 
 such that $n- d_1-d_2< n^\nu$, we have $d_1+d_2>n-n^\nu>\frac{m}{r+1}-
\frac{m}{2(r+1)(r+2)}=\frac{m(2r+3)}{2(r+1)(r+2)}$, and hence 
one of  $(d_1,d_2)$, $(d_2,d_1)$ (or both if $d_1=d_2$) satisfies the conditions of
Lemma~\ref{lem:divat}~(b). 
Thus, by Lemma~\ref{lem:divat}~(b), it follows that if $d_1 \le d_2$,
then either $(d_0,d):=(d_1, d_2)$ satisfies the conditions of
Definition~\ref{def:kr}(c), or $d_2=\frac{m}{r+1}$ and $d_1\leq
\frac{m}{2(r+1)(r+2)}$. Let $P_2'(n,m)$ denote the contribution to
$P_2(n,m)$ from all the pairs $(d_1,d_2)$ where $\{d_1,d_2\}=\{
\frac{m}{r+1},d_0\}$ and $d_0 \leq \frac{m}{2(r+1)(r+2)}$.

For the other pairs, we note that there are $t(r,m,n) \le r+3$
choices for the larger divisor $d$.
Consider a fixed $d\leq \frac{m}{r+1}$, say $d = \frac{m}{c}.$ Then each divisor $d_0$
of $m$, such that $\frac{m}{2(r+1)(r+2)} < d_0 \le d$ and $d + d_0 \le
n$, is equal to $\frac{m}{c_0}$ for some $c_0$ such that
$c \le c_0 < 2(r+1)(r+2)$. Let $d_0(d) = \frac{m}{c_0}$ be the largest of
these divisors $d_0.$ By 
Lemma~\ref{lem:divat}(b), the combined contribution to $P_2(n,m)$
from the ordered pairs $(d,d_0(d))$ and $(d_0(d),d)$ is (since $d$ and
$d_0(d)$ may be equal) at most 
$$
\frac{1}{dd_0(d)} < \frac{2r+3}{m} \cdot \frac{2(r+1)(r+2)}{m} =
\frac{2(r+1)(r+2)(2r+3)}{m^2}.
$$ 
(Note that $\frac{1}{dd_0(d)} \ge \frac{(r+1)^2}{m^2} > \frac{1}{n^2}$.) If 
$d_0=\frac{m}{c'}$ is any other divisor of this type and  
$d_0 < d_0(d)$, then $c_0+1 \le c' < 2(r+1)(r+2)$, and so 
$n-d-d_0=(n-d-d_0(d))+d_0(d)-d_0$ is at least  
$$
d_0(d)-d_0=\frac{m}{c_0} - \frac{m}{c'} \ge\frac{m}{c_0} - \frac{m}{c_0+1}=
\frac{m}{c_0(c_0+1)} > \frac{m}{4(r+1)^2(r+2)^2}.
$$
By Lemma~\ref{newPs}, the contribution to $P_2(n,m)$ from the pairs $(d,d_0)$ and
$(d_0,d)$ is  $O( \frac{1}{m^2}\cdot
\frac{m^{2s+\delta}}{m^3}) = O(\frac{1}{n^{5-2s-\delta}})$. Since there
are $t(r,m,n) \le r+3$ choices for $d$, and a bounded number of divisors $d_0$ 
for a given $d$,  the contribution to $P_2(n,m)$ from
all  the pairs  $(d_1,d_2)$ such that $n- d_1-d_2< n^\nu$
is at  most 
\[
P_2'(n,m) + t(r,m,n) \frac{2(r+1)(r+2)(2r+3)}{n^2r^2}+ O\left(\frac{1}{n^{5-2s-\de}}
\right).
\]
Thus
\begin{eqnarray*}
P_2(n,m)&\le& P_2'(n,m) + \frac{2t(r,m,n)(r+1)(r+2)(2r+3)}{n^2r^2}+
  O\left(\frac{1}{n^{x_2}}\right)  \\
&=&  P_2'(n,m) +\frac{k(r,m,n)}{n^2} + O\left(\frac{1}{n^{x_2}}\right)
\end{eqnarray*}
with $x_2=\min\{1+2s-2\de,5-2s-\de\}=1+2s-2\de$.
Note that
\[
k(r,m,n)\leq (r+3)
\frac{2(r+1)(r+2)(2r+3)}{r^2}=4r^2+30r+80+\frac{90}{r}+\frac{36}{r^2}
\]
which is less than $\frac{4(r+3)^4}{r^2}$. 

Finally we estimate $P_1(n,m)+P'_2(n,m)$. By Lemma~\ref{lem:qi}, $P_1(n,m)=
\sum \frac{1}{d}P_0(n-d,m)$, where the summation is over all $s$-large 
divisors $d$
of $m$ such that $d\leq n$, and we take $P_0(0,m)=1$. 
Note that $d\leq n\leq \frac{m}{r}$, so each divisor $d=\frac{m}{c}$ for some 
$c\geq r$. In the case where 
$m=rn$, that is, the case where $n$ divides $m$ (and only in this case), 
we have a contribution to $P_1(n,m)$ of $\frac{1}{n}$ due to $n$-cycles.
If $d<n$ then $d=\frac{m}{c}$ with $c\geq r+1$.

Next we consider all divisors $d$ of $m$ such that $d\leq \frac{m}{r+2}$. 
For each of these divisors, $n-d\geq n - \frac{m}{r+2}\ge n-\frac{(r+1)n}{r+2} 
=\frac{n}{r+2}$. Thus by Lemma~\ref{newPs}, $P_0(n-d,m)
=  O\left(\frac{m^{2s + \delta}}{n^{3}}\right)
= O\left(\frac{1}{n^{3-2s-\delta}}\right)$.
The number of $d$ satisfying $d\geq \frac{m}{2(r+1)}$ is bounded in terms of 
$r$ (which is fixed), and hence  the contribution to $P_1(n,m)$ 
from all the divisors $d$ satisfying $\frac{m}{2(r+1)}\leq d\leq  \frac{m}{r+2}$ 
is at most $O\left(\frac{1}{m}\,\frac{1}{n^{3-2s-\delta}}\right)=O\left(
\frac{1}{n^{4-2s-\delta}}\right)$. On the other hand, if $m^s\leq d
<\frac{m}{2(r+1)}$, then $n-d>n - \frac{(r+1)n}{2(r+1)} =\frac{n}{2}$. Now 
since $r$ is fixed and $s<1$, for sufficiently large $n$, we have $m^s<\frac{n}
{4}$, and so $n-d> m^s +\frac{n}{4}$. Then, by Lemma~\ref{newPs} (applied 
with $a=1$ and $c=\frac{1}{4}$), 
$P_0(n-d,m)=  O\left(\frac{m^{2s + 2\delta}}{(n-d)^{4}}\right)
= O\left(\frac{1}{n^{4-2s-2\delta}}\right)$, and the contribution to 
$P_1(n,m)$ from all $s$-large divisors $d<  \frac{m}{2(r+1)}$ is at most 
$\frac{d(m)}{m^s}O\left(\frac{1}{n^{4-2s-2\delta}}\right)=
O\left(\frac{1}{n^{4-s-3\delta}}\right)$. Thus, noting that $\min\{4-2s-\de,
4-s-3\de\}\geq 1+2s-2\de$, the contribution to $P_1(n,m)$ from all 
$s$-large divisors $d$ of $m$ such that $d\leq\frac{m}{r+2}$ is 
$O\left(\frac{1}{n^{1+2s-2\delta}}\right)$.

By  Lemma~\ref{lem:divat},  the only  divisor  not  yet considered  is
$d=\frac{m}  {r+1}$ and  this case  of course  arises only  when $r+1$
divides  $m$. Suppose then that $r+1$ divides $m$. We must estimate 
the contribution to $P_1(n,m)+P'_2(n,m)$ from elements containing a 
cycle of length $d=\frac{m}{r+1}$. The  contribution  to  $P_1(n,m)+P'_2(n,m)$ 
due  to  the  divisor $d=\frac{m}{r+1}$    
is   $\frac{r+1}{m}P_0(n-\frac{m}{r+1},m)+\frac{r+1}{m}\sum_{d_0}\frac{1}{d_0} 
P_0(n-\frac{m}{r+1}-d_0,m)$, where the summation is over all $s$-large $d_0\leq 
\frac{m}{2(r+1)(r+2)}$. 
Suppose first that $n=\frac{m}{r+1}\geq \frac{m}{2(r+1)(r+2)-1}$, so that for each $d_0$, 
$n-\frac{m}{r+1}-d_0>\frac{m}{2(r+1)^2(r+2)^2}$.
Then, by Lemma~\ref{newPs}, the contribution to $P_1(n,m)+P'_2(n,m)$ is at most
\[
O\left(\frac{1}{m}.\frac{m^{2s+\de}}{m^{3}}\right)
+d(m) O\left(\frac{1}{m^{1+s}}.\frac{m^{2s+\de}}{m^{3}}\right)
=O\left(\frac{1}{n^{4-2s-\de}}\right)
\]
and this is $ O\left(\frac{1}{n^{1+2s-2\de}}\right)$
since $4-2s-\de\geq 1+2s-2\de$.  Finally suppose that
$n-\frac{m}{r+1}  < \frac{m}{2(r+1)(r+2)}$. In this case we  estimate the  contribution to
$P_1(n,m)+P'_2(n,m)$   from  $d=\frac{m}{r+1}$ by the proportion $\frac{1}{d}=\frac{r+1}{m}$
of elements of $S_n$ containing a $d$-cycle
(recognising  that this  is usually  an over-estimate).  Putting these
estimates together we have
\[
P_1(n,m)+P'_2(n,m)\leq\frac{\alpha}{n}+\frac{\alpha'.(r+1)}{m}+
O\left(\frac{1}{n^{1+2s-2\de}}\right),
\]
where $\alpha=1$ if $m=rn$ and is $0$ otherwise, and $\alpha'=1$ if
$r+1$ divides $m$ and $n-\frac{m}{r+1}<\frac{m}{2(r+1)(r+2)-1}$, and is 0
otherwise.  
The result now follows using (\ref{eq-qi}) and the estimates we have obtained
for each of the summands.
\end{proof}

It is sometimes useful to separate out the results of 
Proposition~\ref{prop:general} according to the values of $m,n$. We do this in 
the theorem below, and also obtain in parts (a) and (b) exact asymptotic 
expressions for $P(n,rn)$ and $P(n,t!(n-t))$ where $r, t$ are bounded and $n$ 
is sufficiently large. For this it is convenient to define two sets of integer pairs.

\begin{definition}\label{T}{\rm
For positive integers $r$ and $m$, define the following sets of integer pairs:
\[
\mathcal{T}(r)=\{(i,j)\,|\, 1\leq i,j\leq r^2, ij
=r^2,\ \mbox{and both}\ r+i, r+j\ \mbox{divide}\ m\}
\]
and $\mathcal{T}'(r)=\{(i,j)\,|\, 1< i,j\leq (r+1)^2, 
(i-1)(j-1) =(r+1)^2,$ and both $r+i, r+j\ \mbox{divide}\ m\}.
$
}
\end{definition}

\begin{theorem}\label{rn}
Let $n,m,r$ be positive integers such that $rn\leq m<(r+1)n$. Let 
$1/2<s\leq 3/4$ and $0<\de\leq s-1/2$. Then, the 
following hold for $r$ fixed and sufficiently large $n$ (where the sets 
$\mathcal{T}(r)$ and $\mathcal{T}'(r)$ are as in Definition~{\rm \ref{T}}).
\begin{enumerate}
\item[(a)] If $m=rn$,  then ${\displaystyle P(n,m)=\frac{1}{n}+\frac{c(r)}{n^2}
+O\left(\frac{1}{n^{1+2s-2\de}}\right)}$, where \\
${\displaystyle
c(r)=\sum_{(i,j)\in\mathcal{T}(r)}(1+\frac{i+j}{2r}).}
$
In particular $c(1)=0$ if $n$ is odd, and $2$ if $n$ is even.
\item[(b)] If $r=t!-1$ and $m=t!(n-t)=(r+1)n-t\cdot t!$, then \\
${\displaystyle
  P(n,m)=\frac{1}{n-t}+\frac{c'(r)}{(n-t)^2}+O\left(\frac{1}{n^{1+2s-2\de}} 
\right)},$ 
where \\
${\displaystyle c'(r)=\sum_{(i,j)\in\mathcal{T}'(r)}(1+\frac{i+j-2}{2(r+1)})}$.
\item[(c)] If $rn<m$, then 
${\displaystyle P(n,m)\leq \frac{\alpha'.(r+1)}{m}+\frac{k(r,m,n)}
{n^2}+ O\left(\frac{1}{n^{1+2s-2\de}}\right)}$, where $\alpha'$ and
$k(r,m,n)$ are as in 
Definition~{\rm \ref{def:kr}}.
\end{enumerate}
\end{theorem}

\begin{proof} 
Part (c) follows immediately from Proposition~\ref{prop:general}.
Next we prove part (a). Suppose that $m=rn$. 
If $r+1$ divides $m$ then we have $n-\frac{m}{r+1}=
\frac{m}{r(r+1)}>\frac{m}{2(r+1)(r+2)-1}$. It follows from 
Proposition~\ref{prop:general} that $P(n,m)\leq\frac{1}{n}+\frac{k(r,m,n)}
{n^2}+O\left(\frac{1}{n^{1+2s-2\de}}\right)$. To complete the proof we refine 
the argument given in the proof of Proposition~\ref{prop:general} for 
$P_2(n,m)$ which gave rise to the term $\frac{k(r,m,n)}{n^2}$. The elements 
contributing to this term were those 
with exactly two $s$-large cycles, where one of these cycles had length 
$d=\frac{m}{r+i}$ for some $i$ such that $1\leq 
i\leq r+3$, and the other had length $d_0(d)=\frac{m}{r+j}$ for some $j$ such 
that $r+i\leq r+j <
2(r+1)(r+2)$ and $d + d_0(d) \le n.$ Moreover, for a given value of $d$,
the  value of $d_0(d)$  was the largest integer with these properties.
Since we now assume that 
$m=rn$ we have
\[
d+d_0(d)=\frac{m(2r+i+j)}{(r+i)(r+j)}\leq n=\frac{m}{r}
\]
that is, $r(2r+i+j)\leq(r+i)(r+j)$, which is equivalent to $r^2\leq ij$.
If $d+d_0(d)$ is strictly less than $n$, that is to say, if $r^2<ij$, and thus
$ij-r^2\geq1$, then
\[
n-d-d_0(d)=n-\frac{rn(2r+i+j)}{(r+i)(r+j)}=\frac{n(ij-r^2)}{(r+i)(r+j)}\geq
\frac{n}{(r+i)(r+j)},
\]
and since $i\leq r+3$ and $r+j<2(r+1)(r+2)$ we have $\frac{n}{(r+i)(r+j)}
\geq \frac{n}{2(r+1)(r+2)(2r+3)}$.
It now follows from Lemma~\ref{newPs} that the contribution to 
$P_2(n,m)$ from all ordered pairs 
$(d,d_0(d))$ and $(d_0(d),d)$ with $d,d_0(d)$ as above and $n>d+d_0(d)$ is $O\left( 
\frac{1}{n^2}\,\frac{m^{2s+\de}}{n^3}\right)=O\left(\frac{1}{n^{5-2s-\de}}
\right)\leq O\left(\frac{1}{n^{1+2s-2\de}}\right)$.
Thus when $m=rn$, the only contributions to the $O\left(\frac{1}{n^2}\right)$ 
term come from pairs
$(\frac{m}{r+i},\frac{m}{r+j})$ such that $r^2=ij$ and $1\leq i,j\leq 
r^2$. (Note that we no longer assume $i\leq j$.) 
These are precisely the pairs $(i,j)\in\mathcal{T}(r)$. For such a pair
$(\frac{m}{r+i},\frac{m}{r+j})$, the contribution to $P_2(n,m)$ is 
\[
\frac{1}{2}\cdot\frac{r+i}{m}\cdot\frac{r+j}{m}=
\frac{r^2+r(i+j)+ij}{2n^2r^2}=\frac{1}{n^2}(1+\frac{i+j}{2r})
\]
(since $ij=r^2$). Thus $P(n,m)\leq\frac{1}{n}+\frac{c(r)}{n^2}
+O\left(\frac{1}{n^{1+2s-2\de}}\right)$.  Moreover, for each 
$(i,j)\in\mathcal{T}(r)$, each permutation in $S_n$ having exactly two cycles 
of lengths $\frac{m}{r+i}$ and $\frac{m}{r+j}$ is a permutation of order 
dividing 
$m$. Thus $P(n,rn)\geq \frac{1}{n}+\frac{c(r)}{n^2}$, and the main assertion 
of part (a) is proved. Finally we note that, if $r=1$ then the only 
possible pair in $\mathcal{T}(1)$ is $(1,1)$, and for this pair to lie in the 
set we require that $r+1=2$ divides $m=n$. Thus $c(1)$ is 0 if $n$ is odd, 
and is 2 if $n$ is even.

Finally we prove part (b) where we have $r=t!-1$ and $m=t!(n-t)$. Then
$rn=(t!-1)n=m+t\cdot t!-n$ which is less than $m$ if $n>t\cdot t!$. Also
$(r+1)n=t!\,n>m$. Thus, for sufficiently large $n$, we have $rn<m<(r+1)n$.
Moreover, $r+1$ divides $m$ and $n-\frac{m}{r+1}=n-(n-t)=t$, which for 
sufficiently large $n$ is less 
than $\frac{n-t}{3t!}<\frac{m}{2(r+1)(r+2)-1}$. It now follows from part (c) 
that $P(n,t!(n-t))\leq \frac{1}{n-t}+\frac{k(r,m,n)}{n^2}+O\left(\frac{1}
{n^{1+2s-2\de}}\right)$. 
Our next task is to improve the coefficient of the $O(\frac{1}{n^2})$ term 
using a similar argument to the proof of part (a).
The elements 
contributing to this term have exactly two $s$-large cycles of lengths 
$d=\frac{m}{r+i}$ and $d_0(d)=\frac{m}{r+j}$, with $r+i,r+j\leq (r+1)(r+2)$ and
\[
d+d_0(d)=\frac{m(2r+i+j)}{(r+i)(r+j)}\leq n=\frac{m}{r+1}+t.
\]
This is equivalent to $(r+1)(2r+i+j)\leq(r+i)(r+j)+\frac{t(r+1)(r+i)(r+j)}{m}$,
and hence, for sufficiently large $n$ (and hence sufficiently large $m$),
$(r+1)(2r+i+j)\leq (r+i)(r+j)$. This is equivalent to $(i-1)(j-1)\geq (r+1)^2$.
If $(i-1)(j-1)> (r+1)^2$, then
\begin{eqnarray*}
n-d-d_0(d)&=&(t+\frac{m}{r+1}) - \frac{m(2r+i+j)}{(r+i)(r+j)}\\
&=&t+\frac{m((i-1)(j-1)-(r+1)^2)}{(r+1)(r+i)(r+j)}\\
&>&\frac{rn}{(r+1)^3(r+2)^2}.
\end{eqnarray*}
As for part (a), the contribution to 
$P_2(n,m)$ from all pairs 
$(\frac{m}{r+i},\frac{m}{r+j})$ with $(i-1)(j-1)> (r+1)^2$ is 
$O\left(\frac{1}{n^{1+2s-2\de}}\right)$.
Thus the only contributions to the $O\left(\frac{1}{n^2}\right)$ 
term come from pairs
$(d,d_0(d))=(\frac{m}{r+i},\frac{m}{r+j})$ such that $(r+1)^2=(i-1)(j-1)$ 
and $1\leq i,j\leq (r+1)^2$.  
These are precisely the pairs $(i,j)\in\mathcal{T}'(r)$. For each of 
these pairs we have $r^2+2r=ij-i-j$ and the contribution to $P_2(n,m)$ is 
\begin{eqnarray*}
\frac{1}{2dd_0(d)}&=&\frac{(r+i)(r+j)}{2m^2}=
\frac{r^2+r(i+j)+ij}{2(r+1)^2(n-t)^2}\\
&=&\frac{(r+1)(2r+i+j)}{2(r+1)^2(n-t)^2}=
\frac{1}{(n-t)^2}\left(1+\frac{i+j-2}{2(r+1)}\right).
\end{eqnarray*}
Thus $P(n,m)\leq\frac{1}{n-t}+\frac{c'(r)}{n^2}
+O\left(\frac{1}{n^{1+2s-2\de}}\right)$. On the 
other hand, each permutation in $S_n$ that contains an $(n-t)$-cycle has 
order dividing $t!(n-t)=m$, and the proportion of these elements is 
$\frac{1}{n-t}$. Also, for each 
$(i,j)\in\mathcal{T}'(r)$, each permutation in $S_n$ having exactly two cycles 
of lengths $\frac{m}{r+i}$ and $\frac{m}{r+j}$, and inducing any permutation 
on the remaining $n-\frac{m}{r+i}-\frac{m}{r+j}=t$ points, is a permutation 
of order dividing $m=t!(n-t)$, and the proportion of all such elements is 
$\frac{c'(r)}{(n-t)^2}$. Thus $P(n,m)\geq \frac{1}{n-t}+\frac{c'(r)}{(n-t)^2}$,
and the assertion 
of part (b) is proved.
\end{proof}

It is a simple matter now to prove Theorems~\ref{leadingterms} and 
\ref{bounds}.

\begin{proof}[Proof of Theorems~{\rm\ref{leadingterms}} and {\rm\ref{bounds}}]
The first theorem follows from Theorem~\ref{rn}~(a) and (b) on 
setting $s=3/4$ and allowing $\delta \rightarrow 0$. Note that
$\frac{1}{n-t} = \frac{1}{n} + \frac{t}{n^2} + O(\frac{1}{n^3})$ and
$\frac{1}{(n-t)^2} = \frac{1}{n^2} +  O(\frac{1}{n^3})$.
For the second theorem, again we set  $s=3/4$ in
Theorem~\ref{rn}(c). By
Proposition~\ref{prop:general} we have $k(r,m,n) \le \frac{4(r+3)^4}{r^2}$. 
If we define $k(r) = \frac{4(r+3)^4}{r^2}$ the result follows.
\end{proof}

\bigskip
Finally we derive the conditional probabilities in Corollary~\ref{cdnlprobs1}.

\begin{proof}[Proof of Corollary~\rm\ref{cdnlprobs1}]
Let  $r,\, n$ be positive integers with $r$ fixed and $n$
`sufficiently large',   and let $g$ be a uniformly
distributed random element of $S_n$. First set $m = rn.$ 
Let $A$ denote the event that $g$ is an  $n$-cycle, and let $B$ denote the 
event that $g$ has order dividing $m$, so that the probability $\Pr(B)$ 
is $P(n,m)$. Then, by elementary probability theory, we have
\begin{eqnarray*}
\Pr( A \mid B) &= &\frac{\Pr( A \cap B)} {\Pr(B)} = \frac{\Pr( A )}
{\Pr(B)}   
 =  \frac{\frac{1}{n}}{P(n,m)}. \\
\end{eqnarray*}
By Theorem~\ref{leadingterms}, $\frac{1}{n}+\frac{c(r)}{n^2}<P(n,m)=\frac{1}{n}+\frac{c(r)}{n^2}+O\left(\frac{1}
{n^{2.5-o(1)}}\right)$, and hence 
\begin{eqnarray*}
1-\frac{c(r)}{n}-O\left(\frac{1}
{n^{1.5-o(1)}}\right)&\leq& \Pr(A \mid B)
\leq 1-\frac{c(r)}{n}+O\left(\frac{1}
{n^{2}}\right).\\
\end{eqnarray*}

Now suppose that $r=t!-1$ for some integer $t\geq2$, and let $A$ denote the 
event that $g$ contains an $(n-t)$-cycle, so that $\Pr(A)=\frac{1}{n-t}$. 
Then, with $B$ as above for the integer $m:=t!(n-t)$, we have 
\begin{eqnarray*}
\Pr( A \mid B) &= &\frac{\Pr( A \cap B)} {\Pr(B)} = \frac{\Pr( A )}
{\Pr(B)}   
 =  \frac{\frac{1}{n-t}}{P(n,m)}. \\
\end{eqnarray*}
By Theorem~\ref{rn}(b),
$\frac{1}{n-t}+\frac{c'(r)}{(n-t)^2}<P(n,m)=\frac{1}{n-t}+
\frac{c'(r)}{(n-t)^2}+O\left(\frac{1} {n^{2.5-o(1)}}\right)$, and hence  
\begin{eqnarray*}
1-\frac{c'(r)}{n}-O\left(\frac{1}
{n^{1.5-o(1)}}\right)&\leq& \Pr(A \mid B)
\leq 1-\frac{c'(r)}{n}+O\left(\frac{1}
{n^{2}}\right).
\end{eqnarray*}
\end{proof}

\subsection*{This research was supported ARC Discovery Grants DP0209706 and
  DP0557587. The authors thank the referee for carefully reading
 the submitted version and advice on the paper.} 

{\footnotesize
\providecommand{\bysame}{\leavevmode\hbox to3em{\hrulefill}\thinspace}

}

\end{document}